\documentclass{amsart}

\usepackage[pagebackref]{hyperref}
\usepackage{cite}
\usepackage[all]{xy}
\usepackage{enumerate}
\usepackage{mathrsfs}
\usepackage{physics}
\usepackage{color}
\usepackage{tikz}
\usepackage{microtype} 
\usepackage{geometry}
\newtheorem{theorem}{Theorem}[section]
\newtheorem{lemma}[theorem]{Lemma}

\newtheorem{proposition}[theorem]{Proposition}

\theoremstyle{definition}
\newtheorem{definition}[theorem]{Definition}
\newtheorem{example}[theorem]{Example}

\newtheorem{remark}{Remark}[section]

\newtheorem*{note}{Note}

\numberwithin{equation}{section}

\def\RR{\mathbb{R}}
\def\CC{\mathbb{C}}
\def\ii{\sqrt{-1}}
\def\NN{\mathcal{N}}
\def\cc{\widetilde{\nabla }}
\def\ttau{\widetilde{\tau }}
\def\grad{\mathrm{grad\,}}
\def\Hess{\mathrm{Hess\,}}

\def\Div{\mathrm{div\,}}
\def\Ker{\mathrm{Ker\,}}
\def\Im{\mathrm{Im\,}}
\def\tr{\mathrm{tr\,}}

\allowdisplaybreaks[4]

\begin{document}

\title[Hermitian pluriharmonic maps]{Hermitian pluriharmonic maps between almost Hermitian manifolds}


\author{Guangwen Zhao}
\address{School of Mathematics and Statistics, Wuhan University of Technology, Wuhan 430070, China}
\curraddr{}
\email{gwzhao@whut.edu.cn}
\thanks{This work is partially supported by the National Natural Science Foundation of China (12001410)}


\subjclass[2020]{53C43, 32Q60}

\keywords{Hermitian pluriharmonic map, holomorphic map, almost Hermitian manifold}


\dedicatory{}

\begin{abstract}
    In the case where both the domain and target manifolds are almost Hermitian, we introduce the concept of Hermitian pluriharmonic maps. We prove that any holomorphic or anti-holomorphic map between almost Hermitian manifolds is Hermitian pluriharmonic. We also establish some monotonicity formulae for the partial energies of Hermitian pluriharmonic maps into K\"ahler manifolds. As an application, under appropriate assumptions on the growth of the partial energies, some holomorphicity results are proven.
\end{abstract}

\maketitle

\tableofcontents

\section{Introduction} 
\label{sec:introduction}

Harmonic maps from K\"ahler manifolds are a powerful tool for studying the rigidity of K\"ahler manifolds, with their holomorphicity being crucial. In \cite{MR584075,MR658472}, Siu studied the holomorphicity of harmonic maps from compact K\"ahler manifolds into compact K\"ahler manifolds with strongly negative curvature or compact quotients of irreducible symmetric bounded domains. He used his $\partial \bar \partial $-Bochner formula and integration by parts to prove a vanishing theorem, which implies that, in this case, harmonic maps are actually pluriharmonic and some curvature terms of the pull-back complexified tangent bundles vanish. This forces the map to be either holomorphic or anti-holomorphic, provided the map has sufficiently high rank. Later, Sampson \cite{MR833809} extended Siu's work by considering target manifolds as Riemannian manifolds with nonpositive Hermitian curvature, proving that in this case, harmonic maps are also pluriharmonic.

When the domain is a complete non-compact K\"ahler manifold, the $\partial \bar \partial $-Bochner technique combined with integration by parts becomes ineffective for studying the holomorphicity of harmonic maps. In this case, the stress-energy tensor related to partial energy density becomes a powerful tool. Dong \cite{MR3149846} used this tool to establish monotonicity formulae for (pluri)harmonic maps originating from complete K\"ahler manifolds equipped with a special exhaustion function, particularly under certain radial curvature conditions. This led to results on holomorphicity under suitable growth conditions for partial energy. Later, Li \cite{MR3459969} improved Dong's results under one of the radial curvature conditions mentioned above.

Similarly, to study the rigidity of the geometric structure of Hermitian manifolds, one needs to consider the case where the domain manifold is Hermitian. To use harmonic maps more effectively in the study of Hermitian manifolds, as in the K\"ahler case, it is necessary for the harmonic maps to be compatible with the complex structure of the domain manifold. When the domain manifold is a non-K\"ahler Hermitian manifold, ordinary harmonic maps typically do not possess this property. To address this, Jost and Yau \cite{MR1226528} introduced Hermitian harmonic maps from Hermitian manifolds into Riemannian manifolds. These maps are compatible with the complex structure of the domain manifold, meaning that when the target manifold is a K\"ahler manifold, a holomorphic map is necessarily Hermitian harmonic. Therefore, compared with ordinary harmonic maps, Hermitian harmonic maps are more suitable for studying Hermitian manifolds. Later, Hermitian harmonic maps were studied by some geometers, such as Chen \cite{MR1301014}, Ni \cite{MR1718630}, Grunau et al. \cite{MR2115446} and Li and Zhang \cite{MR2296630}. Continuing along Dong's work, Yang et al. \cite{MR3047049} studied the monotonicity formulae and holomorphicity of Hermitian harmonic maps from complete Hermitian manifolds into K\"ahler manifolds. In addition, Liu and Yang \cite{MR3275649} studied several different types of harmonic maps from Hermitian manifolds and provided some existence and holomorphicity results. 

The above mainly introduces the geometry and analysis of (generalized) harmonic maps from K\"ahler or Hermitian manifolds. A natural question arises: \emph{when both the domain and target manifolds are almost Hermitian, how should one define harmonic maps that are compatible with the almost complex structures of both?} Using the second canonical connection, Zhang \cite{MR3077216} defined Hermitian harmonic maps between almost Hermitian manifolds and proved that any (anti-)holomorphic map between almost Hermitian manifolds must be Hermitian harmonic. When the domain manifold is almost K\"ahler and the target manifold is K\"ahler, the Hermitian harmonic map degenerates into a harmonic map. In this case, Li and Zhang \cite{li-zhang} proved a holomorphicity result. They also derived a strong rigidity result for compact almost K\"ahler manifolds using a generalized Bochner formula.

Guided by Zhang \cite{MR3077216} and inspired by the research of some previously mentioned works, this paper defines Hermitian pluriharmonic maps between almost Hermitian manifolds using the second canonical connection (see Section~\ref{sec:basic_concepts}). We prove that any (anti-)holomorphic map between almost Hermitian manifolds must be Hermitian pluriharmonic (see Section~\ref{sec:pluriharmonicity_of_holomorphic_maps}). Additionally, following the work of \cite{MR3149846} and \cite{MR3047049}, we extend and partially refine their results under the same radial curvature conditions by selecting more suitable gradient vector fields $X$ in integral formula \eqref{eq-r}. We establish monotonicity formula for the partial energies of Hermitian pluriharmonic maps from complete almost Hermitian manifolds into K\"ahler manifolds (see Section~\ref{sec:monotonicity_formulae_for_hermitian_pluriharmonic_maps}). Furthermore, under appropriate growth assumptions of the partial energies, we obtain holomorphicity results (see Section~\ref{sec:holomorphicity_of_hermitian_pluriharmonic_maps}).


\section{Basic concepts} 
\label{sec:basic_concepts}

Let $(M^{2d},J)$ be an almost complex manifold with $\dim_{\RR}M=2d$. The Nijenhuis tensor $\NN_J$ of $(M,J)$ (torsion of $J$) is a tensor field of type $(1,2)$ given by
\[
	4\NN_J(X,Y)=[X,Y]+J[X,JY]+J[JX,Y]-[JX,JY],\quad \mbox{for}\ X,Y\in \Gamma(TM).
\]
A Hermitian metric $g$ on $(M,J)$ is a $J$-invariant Riemannian metric, i.e., 
\[
	g(JX,JY)=g(X,Y)
\]
for any $X,Y\in \Gamma(TM)$. An almost Hermitian manifold $(M,J,g)$ is an almost complex manifold $(M,J)$ with a Hermitian metric $g$. For any $p\in M$, denote by $T_pM^\CC=T_pM\otimes_\RR \CC$ the complexified tangent space of $M$ at $p$, and then one has the complexified tangent bundle $TM^\CC$. We can extend the almost complex structure $J$ and the Hermitian metric $g$ to $TM^\CC$ by $\CC$-linearity. The complexified tangent space $T_pM^\CC$ can be decomposed into a direct sum of the eigenspace of $J$: 
\[
	T_pM^\CC=T_p^{1,0}M\oplus T_p^{0,1}M,
\]
where $T_p^{1,0}M=\{X-\ii JX:X\in T_pM\}$ and $T_p^{0,1}M=\overline{T_p^{1,0}M}$. We also have the corresponding bundles $T^{1,0}M$ and $T^{0,1}M$.

Let $\omega(\cdot ,\cdot )=g(J\cdot ,\cdot )$ be the fundamental $(1,1)$-form of the almost Hermitian manifold $(M,J,g)$. Refer to \cite{MR184185}, an almost Hermitian manifold $(M,J,g,\omega )$ is called 
\begin{description}
	\item[K] K\"ahler if $\nabla J=0$;
	\item[AK] almost K\"ahler if $\dd \omega =0$;
	\item[NK] nearly K\"ahler if $(\nabla_XJ)Y+(\nabla_YJ)X=0$ for any $X,Y\in \Gamma(TM)$; 
	\item[QK] quasi-K\"ahler if $(\nabla_XJ)Y+(\nabla_{JX}J)JY=0$ for any $X,Y\in \Gamma(TM)$;
	\item[SK] semi-K\"ahler if $\delta J=0$;
	\item[H] Hermitian if $\NN_J=0$.    
\end{description}
It is also pointed out in \cite{MR184185} that 
\begin{equation}\label{eq-2c}
	\bf K\subset AK\subset QK\subset SK,\quad 
	K\subset AK\subset QK\subset SK,\quad 
	K\subset H,\quad 
	K=H\cap QK=AK\cap NK.
\end{equation}

\subsection{Hermitian pluriharmonic maps between almost Hermitian manifolds} 
\label{sub:hermitian_pluriharmonic_maps}

The second canonical connection $\cc $ on an almost Hermitian manifold $(M,J,g,\omega )$ (cf. \cite{MR1456265}) is given by
\begin{equation}\label{eq-c}
	\begin{split}
		\langle {\cc}_XY,Z\rangle =\langle \nabla_XY,Z\rangle-\frac{1}{2}\langle J(\nabla_XJ)Y,Z\rangle 
        +\frac{1}{4}\left\{\langle X,(\nabla_{JY}J)Z+J(\nabla_YJ)Z\rangle -\langle X,(\nabla_{JZ}J)Y+J(\nabla_ZJ)Y\rangle \right\}
	\end{split}
\end{equation}
for any $X,Y,Z\in \Gamma(TM)$. The second canonical connection is the unique connection satisfying the following three condition
\[
	\cc g=0,\quad \cc J=0,\quad (\tau^{\cc})^{1,1}=0,
\]
where $(\tau^{\cc})^{1,1}$ is the $J$-invariant part of the torsion of $\cc $.

Using the second canonical connections, Zhang \cite{MR3077216} defined the following Hermitian harmonic map.
\begin{definition}[cf. \cite{MR3077216}]\label{defn-a}
	A smooth map $u:(M,J,g)\to (N,J^N,h)$ is called a Hermitian harmonic map between almost Hermitian manifolds, if it satisfies
    \[
	\widetilde{\tau }(u)=0,
    \]
    where $\widetilde{\tau }(u)=\tr_g\cc \dd u
    $ is the Hermitian tension field, and $\cc \dd u(X,Y)={\cc}^N_{\dd u(X)}(\dd u(Y))-\dd u\left(\cc_XY\right)$.
\end{definition}

\begin{remark}\label{rem-a}
	Let $\{e_A\}_{A=1}^{2m}=\{e_\alpha ;Je_\alpha \}_{\alpha =1}^m$ is the local orthonormal frame field on the domain manifold $(M^{2m},g,J)$. Set
	\[
	E_\alpha =\frac{1}{\sqrt{2}}(e_\alpha -\ii Je_\alpha ),\quad \bar{E}_\alpha =\frac{1}{\sqrt{2}}(e_\alpha +\ii Je_\alpha ),
	\]
    then $\{E_\alpha \}_{\alpha =1}^m$ is the local unitary frame field on $M$. Therefore,
    \[
    \ttau(u)=\tr_g\cc \dd u=\sum_{A=1}^{2m}\cc \dd u(e_A,e_A)=\sum_{\alpha =1}^m\left[\cc \dd u(E_\alpha ,\bar{E}_\alpha )+\cc \dd u(\bar{E}_\alpha ,E_\alpha )\right].
    \]
\end{remark}

Inspired by \cite{MR3077216}, we introduce the concept of Hermitian pluriharmonic maps between almost Hermitian manifolds by means of the second canonical connections.
\begin{definition}\label{defn-b}
    A smooth map $u:(M,J,g)\to (N,J^N,h)$ is said to be Hermitian pluriharmonic if $(\cc \dd u)^{1,1}=0$, that is,
    \begin{align}\label{eq-p}
	    \cc \dd u\left(Z,\overline{W}\right)=0
    \end{align}
    for any $Z,W\in \Gamma(T^{1,0}M)$. 	
\end{definition}

\begin{remark}\label{rem-b}
(i) It follows from Definition~\ref{defn-a} and Definition~\ref{defn-b} that any Hermitian pluriharmonic map is Hermitian harmonic. In fact, it is easy to see that equation \eqref{eq-p} is equivalent to 
\begin{align}
	\cc \dd u\left(Z,\overline{W}\right)+\cc \dd u\left(\overline{Z},W\right)=0
\end{align}
for any $Z,W\in \Gamma(T^{1,0}M)$.

(ii) By setting $Z=X-\ii JX,\ W=Y-\ii JY$, it is easy to know that a smooth map $u:(M,J,g)\to (N,J^N,h)$ is Hermitian pluriharmonic if and only if
\begin{equation}\label{eq-b}
	\cc \dd u(X,Y)+\cc \dd u(JX,JY)=0
\end{equation}
for any $X,Y\in \Gamma(TM)$. 
\end{remark}


\subsection{The stress-energy tensors} 
\label{sub:the_stress_energy_tensors}

This part can be found in \cite{MR3149846,MR584075,MR851976}. Let $u:(M,J,g)\to (N,J^N,h)$ be a smooth map between almost Hermitian manifolds. Define
\begin{align*}
	&\partial u:T^{1,0}M\to T^{1,0}N,\quad \bar \partial u:T^{0,1}M\to T^{1,0}N,\\
    &\partial \bar u:T^{1,0}M\to T^{0,1}N,\quad \bar \partial \bar u:T^{0,1}M\to T^{0,1}N
\end{align*}
by 
\[
	\dd u|_{T^{1,0}M}=\partial u+\partial \bar u,\quad \dd u|_{T^{0,1}M}=\bar \partial u+\bar \partial \bar u.
\]
Indeed, there has an orthogonal decomposition
\begin{equation}\label{eq-h}
	\dd u=\frac{1}{2}(\dd u+J^N\circ \dd u\circ J)+\frac{1}{2}(\dd u-J^N\circ \dd u\circ J).
\end{equation}
It is easy to see that $\bar \partial \bar u=\overline{\partial u}$ and $\partial \bar u=\overline{\bar \partial u}$. 
The map $u$ is said to be holomorphic\footnote{Since there are no complex coordinates in this case, some authors prefer to use the terms ``almost holomorphic" or ``pseudo-holomorphic". However, for convenience, we still use the term ``holomorphic" in this paper, which can also be found in \cite[(9.5)]{MR495450}.} (resp. anti-holomorphic) if 
\[
	\dd u\circ J=J^N\circ \dd u\quad (\mbox{resp.}\ \dd u\circ J=-J^N\circ \dd u).
\]
It follows from the previous decomposition that $u$ is holomorphic (resp. anti-holomorphic) if and only if $\bar\partial u=0\ (\mbox{resp.}\ \partial u=0)$. 

For a smooth map $u:(M,J,g)\to (N,J^N,h)$, from the decomposition \eqref{eq-h}, we introduce two $1$-forms $\sigma ,\sigma'\in A^1(u^{-1}TN)$ as follows:
\[
	\sigma(X)=\frac{\dd u(X)+J^N\dd u(JX)}{2}\quad \mbox{and}\quad \sigma'(X)=\frac{\dd u(X)-J^N\dd u(JX)}{2}
\]
for any $X\in \Gamma(TM)$. As in \cite{MR3149846}, one has
\[
	\sigma(JX)=-J^N\sigma(X)\quad \mbox{and}\quad \sigma'(JX)=J^N\sigma'(X).
\]
Thus, both $\sigma $ and $\sigma'$ are $J$-invariant:
\begin{align*}
	\langle \sigma(JX),\sigma(JY)\rangle =\langle \sigma(X),\sigma(Y)\rangle ,\quad 
	\langle \sigma'(JX),\sigma'(JY)\rangle =\langle \sigma'(X),\sigma'(Y)\rangle .
\end{align*}
In addition, it is easy to calculate $|\sigma |^2=2|\bar \partial u|^2, |\sigma' |^2=2|\partial u|^2$ and $|\dd u|^2=|\sigma |^2+|\sigma'|^2$. Hence, the energy density $e(u)=\frac{1}{2}|\dd u|^2=|\bar \partial u|^2+|\partial u|^2$, and then the energy
\begin{align*}
	E(u)=\int_Me(u)=\int_M|\bar \partial u|^2+\int_M|\partial u|^2.
\end{align*}
We call $E''(u)=\int_M|\bar \partial u|^2$ and $E'(u)=\int_M|\partial u|^2$ are the partial energies of $u$, respectively. 

Let $M$ be a Riemannian manifold and let $A^p(E)=\Gamma(\Lambda^pT^*M\otimes E)$ be the space of smooth $p$-form on $M$ with values in the Riemannian vector bundle $\pi :E\to M$. For $\omega \in A^p(E)$, define a symmetric $(0,2)$-tensor field $\omega \odot \omega $ by
\[
	\omega \odot \omega (X,Y)=\langle \iota_X\omega ,\iota_Y\omega \rangle ,\quad X,Y\in \Gamma(TM),
\]
where $\iota_X$ is the interior product by $X$. The stress-energy tensor of $\omega$ is given as follows:
\begin{align*}
	S_\omega=\frac{|\omega |^2}{2}g-\omega \odot \omega .
\end{align*}
Recall that for a $(0,2)$-tensor field $T\in \Gamma(T^*M\otimes T^*M)$, its divergence $\Div T\in \Gamma(T^*M)$ is given by
\[
	(\Div T)(X)=\sum_A(\nabla_{e_A}T)(e_A,X),\quad X\in \Gamma(TM).
\]
We have the following useful lemma.
\begin{lemma}[cf. \cite{MR3149846,MR851976,MR3047049}]\label{lem-d}
    $(\Div S_\omega )(X)=\langle \delta \omega ,\iota_X\omega \rangle +\langle \iota_X\dd \omega ,\omega \rangle $ for any $X\in \Gamma(TM)$.
\end{lemma}
Let $D$ be any bounded domain of $M$ with $C^1$ boundary. We have the following integral formula (cf. \cite{MR3149846,MR1391729})
\begin{align}\label{eq-r}
	\int_{\partial D}S_\omega (X,\nu )\dd v=\int_D\left\{\langle S_\omega ,\nabla X^\flat \rangle +(\Div S_\omega )(X)\right\}\dd v, 
\end{align}
where $\nu $ is the unit outward normal vector field along $\partial D$ in $D$, and $X^\flat$ is the dual $1$-from of $X$ and $\nabla X^\flat $ is given by $(\nabla X^\flat )(Y,Z)=\langle \nabla_YX,Z\rangle $ for $Y,Z\in \Gamma(TM)$.


\begin{note}
\emph{For ease of exposition, this paper focuses only on topics such as the Hermitian pluriharmonicity of holomorphic maps, the monotonicity formula of the partial energy $E''$, and the holomorphicity of Hermitian pluriharmonic maps. Similarly, for topics related to the Hermitian pluriharmonicity of of anti-holomorphic maps, the monotonicity formula of the partial energy $E'$, and the anti-holomorphicity of Hermitian pluriharmonic maps, it is simply required to replace the almost complex structure $J^N$ of the target manifold $N$ with $-J^N$.}

\emph{Throughout this paper, the Levi--Civita connections on $M$ and $N$ are denoted by $\nabla $ and $\nabla^N$, respectively, and $\nabla^u$ the induced connection on $T^*M\otimes u^{-1}TN$. Similarly, the second canonical connections on $M,$ and $N$ are denoted as $\cc$ and $\cc^N$, respectively. We consistently use $\{e_A\}_{A=1}^{2m}=\{e_\alpha ;Je_\alpha \}_{\alpha =1}^m$ to represent a local orthonormal frame field on the domain manifold $(M,J,g)$. Additionally, for convenience, we denote by $\langle \ ,\ \rangle $ the inner products on all bundles.}
\end{note}


\section{Hermitian pluriharmonicity of holomorphic maps} 
\label{sec:pluriharmonicity_of_holomorphic_maps}

In \cite{MR3077216}, Zhang proved that any holomorphic map between almost Hermitian manifolds must be Hermitian harmonic. Furthermore, we can also prove the Hermitian pluriharmonicity of holomorphic maps. Precisely, we have
\begin{theorem}\label{thm-a}
 	Any holomorphic map between almost Hermitian manifolds must be Hermitian pluriharmonic.
\end{theorem} 

To prove this theorem, we first give a simple lemma whose proof is a straightforward calculation. 
\begin{lemma}\label{lem-b}
	Let $(M,J,g)$ be an almost Hermitian manifold and $\nabla $ its Levi--Civita connection. Then
	\[
	J(\nabla_XJ)Y-(\nabla_{JX}J)Y-J(\nabla_YJ)X+(\nabla_{JY}J)X=4\NN_J(X,Y).
	\]
\end{lemma}

\begin{proof}[Proof of Theorem~\ref{thm-a}]
	We only consider the holomorphic case, and the anti-holomorphic case is similar. Let $u:(M,J,g)\to (N,J^N,h)$ be a holomorphic map between almost Hermitian manifolds. For any $X,Y\in \Gamma(TM)$, we have  
    \begin{align}\label{eq-k}
        \begin{split}
    	    \cc \dd u(X,Y)+\cc \dd u(JX,JY)
		    =&{\cc}^N_{\dd u(X)}(\dd u(Y))+{\cc}^N_{\dd u(JX)}(\dd u(JY))
		    -\dd u\left(\cc_XY+\cc_{JX}(JY)\right)\\
	    	=&{\cc}^N_{\dd u(X)}(\dd u(Y))+{\cc}^N_{J^N\dd u(X)}(J^N\dd u(Y))
	    	-\dd u\left(\cc_XY+\cc_{JX}(JY)\right).
        \end{split}
	\end{align}
	From \eqref{eq-c}, a simple calculation yields
	\begin{align}\label{eq-l}
	    \begin{split}
		    \cc_XY+\cc_{JX}(JY)
	        =\nabla_XY+\nabla_{JX}(JY)-\frac{1}{2}J(\nabla_XJ)Y-\frac{1}{2}(\nabla_{JX}J)Y-\frac{1}{2}J(\nabla_YJ)X-\frac{1}{2}(\nabla_{JY}J)X.
	    \end{split}
	\end{align}
	Similarly, we also have
	\begin{align}\label{eq-m}
		\begin{split}
		    {\cc}^N_{\dd u(X)}(\dd u(Y))+{\cc}^N_{J^N\dd u(X)}(J^N\dd u(Y))
	        =&\nabla^N_{\dd u(X)}(\dd u(Y))+\nabla^N_{J^N\dd u(X)}(J^N\dd u(Y))\\
	        &-\frac{1}{2}J^N\left(\nabla^N_{\dd u(X)}J^N\right)\dd u(Y)-\frac{1}{2}\left(\nabla^N_{J^N\dd u(X)}J^N\right)(\dd u(Y))\\
	        &-\frac{1}{2}J^N\left(\nabla^N_{\dd u(Y)}J^N\right)\dd u(X)-\frac{1}{2}\left(\nabla^N_{J^N\dd u(Y)}J^N\right)(\dd u(X)).
	    \end{split}
	\end{align}
	On the other hand, 
    \begin{align}\label{eq-j}
        \begin{split}
    	    \nabla^N_{J^N\dd u(X)}(J^N\dd u(Y))
    	    =&\left(\nabla^N_{J^N\dd u(X)}J^N\right)(\dd u(Y))+J^N\left(\nabla^u_{JX}\dd u\right)(Y)+J^N\dd u\left(\nabla_{JX}Y\right)\\
    	    =&\left(\nabla^N_{J^N\dd u(X)}J^N\right)(\dd u(Y))+J^N\left(\nabla^u_Y\dd u\right)(JX)-J^N\dd u\left(\nabla_{JX}(J^2Y)\right),
        \end{split}
    \end{align}
    where 
    \begin{align*}
    	\begin{split}
    		J^N(\nabla^u_Y\dd u)JX
    		=&J^N\nabla^N_{\dd u(Y)}(\dd u(JX))-J^N\dd u\left(\left(\nabla_{Y}J\right)X\right)-J^N\dd uJ(\nabla_YX)\\
    		=&J^N\left(\nabla^N_{\dd u(Y)}J^N\right)(\dd u(X))-(\nabla^u_Y\dd u)(X)-\dd u\left(\nabla_YX\right)-J^N\dd u\left(\left(\nabla_{Y}J\right)X\right)+\dd u\left(\nabla_YX\right)\\
    		=&J^N\left(\nabla^N_{\dd u(Y)}J^N\right)(\dd u(X))-(\nabla^u_X\dd u)(Y)-J^N\dd u\left(\left(\nabla_{Y}J\right)X\right)\\
    		=&J^N\left(\nabla^N_{\dd u(Y)}J^N\right)(\dd u(X))-\nabla^N_{\dd u(X)}(\dd u(Y))+\dd u\left(\nabla_XY\right)-\dd u\left(J\left(\nabla_{Y}J\right)X\right)
    	\end{split}
    \end{align*}
    and
    \begin{align*}
    	J^N\dd u\left(\nabla_{JX}(J^2Y)\right)
    	=J^N\dd u\left(\left(\nabla_{JX}J\right)JY\right)+J^N\dd u\left(J\left(\nabla_{JX}(JY)\right)\right)
    	=\dd u\left(J\left(\nabla_{JX}J\right)JY\right)-\dd u\left(\nabla_{JX}(JY)\right).
    \end{align*}
    Hence \eqref{eq-j} becomes 
    \begin{align}\label{eq-n}
        \begin{split}
    	    \nabla^N_{\dd u(X)}(\dd u(Y))+\nabla^N_{J^N\dd u(X)}(J^N\dd u(Y))
    	    =&\left(\nabla^N_{J^N\dd u(X)}J^N\right)(J^N\dd u(Y))+J^N\left(\nabla^N_{\dd u(Y)}J^N\right)(\dd u(X))+\dd u\left(\nabla_XY\right)\\
    	    &-\dd u\left(J(\nabla_YJ)X\right)-\dd u\left(J(\nabla_{JX}J)Y\right)+\dd u\left(\nabla_{JX}(JY)\right).
    	\end{split}
    \end{align}
    Substituting \eqref{eq-n} into \eqref{eq-m}, we obtain
    \begin{align}\label{eq-o}
		\begin{split}
		    &{\cc}^N_{\dd u(X)}(\dd u(Y))+{\cc}^N_{J^N\dd u(X)}(J^N\dd u(Y))\\
	        =&\frac{1}{2}\left(\nabla^N_{J^N\dd u(X)}J^N\right)(\dd u(Y))-\frac{1}{2}J^N\left(\nabla^N_{\dd u(X)}J^N\right)(\dd u(Y))\\
	        &+\frac{1}{2}J^N\left(\nabla^N_{\dd u(Y)}J^N\right)(\dd u(X))-\frac{1}{2}\left(\nabla^N_{J^N\dd u(Y)}J^N\right)(\dd u(X))\\
	        &+\dd u\left(\nabla_XY\right)-\dd u\left(J(\nabla_YJ)X\right)-\dd u\left(J(\nabla_{JX}J)Y\right)+\dd u\left(\nabla_{JX}(JY)\right).
	    \end{split}
	\end{align}
	Substituting \eqref{eq-l} and \eqref{eq-o} into \eqref{eq-k}, we derive
	\begin{align*}
        \begin{split}
    	    &\cc \dd u(X,Y)+\cc \dd u(JX,JY)\\
    	    =&\frac{1}{2}\dd u\left\{J(\nabla_XJ)Y-(\nabla_{JX}J)Y-J(\nabla_YJ)X+(\nabla_{JY}J)X\right\}\\
		    &-\frac{1}{2}\left\{J^N\left(\nabla^N_{\dd u(X)}J^N\right)(\dd u(Y))-\left(\nabla^N_{J^N\dd u(X)}J^N\right)(\dd u(Y))-J^N\left(\nabla^N_{\dd u(Y)}J^N\right)(\dd u(X))+\left(\nabla^N_{J^N\dd u(Y)}J^N\right)(\dd u(X))\right\}.
        \end{split}
	\end{align*}
	From Lemma~\ref{lem-b}, this equation reduces to
    \begin{align*}
    	\cc \dd u(X,Y)+\cc \dd u(JX,JY)
    	=2\left\{\dd u(\NN_J(X,Y))-\NN_{J^N}(\dd u(X),\dd u(Y))\right\}.
    \end{align*}
    However, it is obvious that $\dd u(\NN_J(X,Y))=\NN_{J^N}(\dd u(X),\dd u(Y))$ since $u$ is holomorphic. Hence the conclusion that $u$ is Hermitian pluriharmonic follows from Remark~\ref{rem-b}, (ii).
\end{proof}


\section{Monotonicity formulae of Hermitian pluriharmonic maps} 
\label{sec:monotonicity_formulae_for_hermitian_pluriharmonic_maps} 

In this section, we use integral formula \eqref{eq-r}, where $\omega =\sigma $, to derive the monotonicity formula for the partial energy $E''$ of Hermitian pluriharmonic maps into a K\"ahler manifold. To this end, we need some preliminary work to estimate each integrand in \eqref{eq-r}.

\begin{proposition}\label{prop-a}
	Let $u:(M,J,g)\to (N,J^N,h)$ be a smooth map from an almost Hermitian manifold to a K\"ahler manifold. Then $u$ is a Hermitian pluriharmonic map if and only if it satisfies 
	\begin{equation}\label{eq-e}
		\nabla \dd u(X,Y)+\nabla \dd u(JX,JY)+\dd u(\alpha(X,Y))=0,
	\end{equation}
	where $\nabla \dd u$ is the second fundamental form in the Levi--Civita connections of $M$ and $N$, and 
	\begin{align*}
	\alpha(X,Y)=\frac{1}{2}J\left\{(\nabla_XJ)Y+(\nabla_{JX}J)JY+(\nabla_YJ)X+(\nabla_{JY}J)JX\right\}.
	\end{align*}
\end{proposition}
\begin{proof}
	It is evident that the K\"ahler condition ($\nabla^NJ^N=0$) implies that the second canonical connection $\cc^N$ is consistent with the Levi--Civita connection $\nabla^N$. For any $X,Y\in \Gamma(TM)$, we have
    \begin{align*}
	    {\cc}\dd u(X,Y)
	    =\nabla^N_{\dd u(X)}(\dd u(Y))-\dd u\left(\cc_XY\right).
    \end{align*}
    Therefore,
    \begin{align}\label{eq-f}
        \begin{split}
        	{\cc}\dd u(X,Y)+{\cc}\dd u(JX,JY)
        	=\nabla^N_{\dd u(X)}(\dd u(Y))+\nabla^N_{\dd u(JX)}(\dd u(JY))-\dd u\left(\cc_XY+\cc_{JX}(JY)\right).
        \end{split}
    \end{align}
    By $\langle ,\rangle $ is $J$-invariant, that is, $J$ is skew self-adjoint, and it is easy to know that $\nabla_\star J$ is skew self-adjoint. In addition, we also have $J(\nabla_\star J)=-(\nabla_\star J)J$. Hence, a simple calculation yields
    \begin{align}\label{eq-g}
	    \cc_XY+\cc_{JX}(JY)
	    =\nabla_XY+\nabla_{JX}(JY)-\alpha(X,Y).
    \end{align}
    Substituting \eqref{eq-g} into \eqref{eq-f}, we derive that
    \begin{align}\label{eq-i}
        \begin{split}
        	{\cc}\dd u(X,Y)+{\cc}\dd u(JX,JY)
        	=\nabla \dd u(X,Y)+\nabla \dd u(JX,JY)+\dd u(\alpha(X,Y)).
        \end{split}
    \end{align}
    Combining with Remark~\ref{rem-b}, (ii), we complete the proof. 
\end{proof}

By taking $X=Y=e_\alpha $ in equation~\eqref{eq-i} and summering it over $i$, we obtain
\begin{align*}
	\ttau (u)
	=\tau (u)+\dd u\left(\sum_{\alpha =1}^m\alpha(e_\alpha ,e_\alpha )\right)=\tau (u)+\dd u(-J\delta J),
\end{align*}
where $\delta $ is the co-differential operator, which is the adjoint of $\dd $ with respect to $\int_M\langle ,\rangle \dd v_g$. Hence we have
\begin{proposition}\label{prop-b}
	Let $u:(M,J,g)\to (N,J^N,h)$ be a smooth map from an almost Hermitian manifold to a K\"ahler manifold. Then $u$ is a Hermitian harmonic map if and only if it satisfies
    \begin{align}\label{eq-s}
    	\tau (u)+\dd u(V)=0,
    \end{align}
    where $V=-J\delta J$.
\end{proposition}

\begin{remark}\label{rem-c}
	Let $V$ be a smooth vector field on a Riemannian manifold $M$. A smooth map $u:M\to N$ between Riemannian manifolds is said to be $V$-harmonic, if it satisfies equation \eqref{eq-s} (cf. \cite{MR2995205,MR3427131}). Therefore, any Hermitian harmonic map from an almost Hermitian manifold $(M,J,g)$ to a K\"ahler manifold is $V$-harmonic with $V=-J\delta J$.
\end{remark}

\begin{lemma}\label{lem-a}
	If $u:(M,J,g)\to (N,J^N,h)$ be a Hermitian pluriharmonic map from an almost Hermitian manifold to a K\"ahler manifold, then $\delta \sigma =\sigma(V)$.
\end{lemma}
\begin{proof}
	By the definition of $\delta $, we have
	\begin{align*}
		\delta \sigma =&-\sum_A(\nabla^u_{e_A}\sigma )(e_A)\\
		=&-\sum_A\nabla^N_{\dd u(e_A)}(\sigma (e_A))+\sum_A\sigma\left(\nabla_{e_A}e_A\right)\\
		=&-\frac{1}{2}\sum_A\nabla^N_{\dd u(e_A)}(\dd u(e_A))-\frac{1}{2}\sum_A\nabla^N_{\dd u(e_A)}(J^N\dd uJ(e_A))\\
		&+\frac{1}{2}\dd u\left(\sum_A\cc_{e_A}e_A\right)+\frac{1}{2}J^N\dd uJ\left(\sum_A\cc_{e_A}e_A\right)+\sigma\left(\sum_A\left(\nabla_{e_A}e_A-\cc_{e_A}e_A\right)\right)\\
		=&-\frac{1}{2}\sum_A\cc \dd u(e_A, e_A)-\frac{1}{2}J^N\sum_A\cc \dd u(e_A,Je_A)+\sigma\left(\sum_A\left(\nabla_{e_A}e_A-\cc_{e_A}e_A\right)\right),
	\end{align*}
	where in the last step we used $\cc J=\nabla^NJ^N=0$. However, the assumption that $u$ is a Hermitian pluriharmonic map (and then a Hermitian harmonic map) implies that $\sum\limits_A\cc \dd u(e_A,Je_A)=\sum\limits_i\left\{\cc \dd u(e_\alpha ,Je_\alpha )+\cc \dd u(Je_\alpha ,J^2e_\alpha )\right\}=0$ and $\sum\limits_A\cc \dd u(e_A,e_A)=0$.
	On the other hand, by \eqref{eq-c}, for any $Z\in \Gamma(TM)$, we have
	\begin{align*}
	 	\left \langle \sum_A\left(\nabla_{e_A}e_A-\cc_{e_A}e_A\right),Z\right \rangle 
	 	=&\frac{1}{2}\sum_A\langle J(\nabla_{e_A}J)e_A,Z\rangle \\
	 	&-\frac{1}{4}\sum_A\langle e_A,(\nabla_{Je_A}J)Z+J(\nabla_{e_A}J)Z\rangle +\frac{1}{4}\sum_A\langle e_A,(\nabla_{JZ}J)e_A+J(\nabla_ZJ)e_A\rangle .
	\end{align*} 
	By repeatedly using the property of the skew self-adjoint of $J$ and $\nabla_\star J$, and noting that $J(\nabla_\star J)=-(\nabla_\star J)J$, we can get
	\begin{align*}
    &\langle e_A,(\nabla_{Je_A}J)Z\rangle =-\langle J(\nabla_{Je_A}J)Je_A,Z\rangle ,\\
    &\langle e_A,J(\nabla_{e_A}J)Z\rangle =-\langle J(\nabla_{e_A}J)e_A,Z\rangle 
	\end{align*}
	and
	\[
	\langle e_A,(\nabla_{JZ}J)e_A\rangle =\langle e_A,J(\nabla_ZJ)e_A\rangle =0.
	\]
	It follows that
	\[
    \sum_A\left(\nabla_{e_A}e_A-\cc_{e_A}e_A\right)=\sum_AJ(\nabla_{e_A}J)e_A=-J\delta J=V.
	\]
	So in summary, we derive that $\delta \sigma =\sigma (V)$.
\end{proof}

\begin{lemma}\label{lem-c}
	If $u:(M,J,g)\to (N,J^N,h)$ be a Hermitian pluriharmonic map from an almost Hermitian manifold to a K\"ahler manifold, then $\dd \sigma =J^N\circ \dd u\circ \NN_J$.
\end{lemma}
\begin{proof}
	For any $X,Y\in \Gamma(TM)$, we have
	\begin{align*}
		\dd \sigma (X,Y)=&\nabla^N_{\dd u(X)}\left(\sigma(Y)\right)-\nabla^N_{\dd u(Y)}\left(\sigma(X)\right)-\sigma([X,Y])\\
		=&\frac{1}{2}\nabla^N_{\dd u(X)}\left(\dd u(Y)\right)+\frac{1}{2}\nabla^N_{\dd u(X)}(J^N\dd uJ(Y))-\frac{1}{2}\nabla^N_{\dd u(Y)}\left(\dd u(X)\right)-\frac{1}{2}\nabla^N_{\dd u(Y)}(J^N\dd uJ(X))\\
		&-\frac{1}{2}\dd u\left(\nabla_XY-\nabla_YX\right)-\frac{1}{2}J^N\dd uJ\left(\nabla_XY-\nabla_YX\right)\\
		=&\frac{1}{2}\nabla^N_{\dd u(X)}(J^N\dd uJ(Y))-\frac{1}{2}\nabla^N_{\dd u(Y)}(J^N\dd uJ(X))-\frac{1}{2}J^N\dd uJ(\nabla_XY)+\frac{1}{2}J^N\dd uJ(\nabla_YX)\\
		=&\frac{1}{2}J^N\nabla \dd u(X,JY)-\frac{1}{2}J^N\nabla \dd u(JX,Y)+\frac{1}{2}J^N\dd u(\nabla_XJ)Y-\frac{1}{2}J^N\dd u(\nabla_YJ)X.
	\end{align*}
	Equation \eqref{eq-e} gives 
	\begin{align*}
	    \nabla\dd u(X,JY)-\nabla\dd u(JX,Y)
	    =&-\dd u(\alpha (X,JY))\\
	    =&-\frac{1}{2}\dd uJ\left\{(\nabla_XJ)JY-(\nabla_{JX}J)Y+(\nabla_{JY}J)X-(\nabla_YJ)JX\right\}\\
	    =&-\frac{1}{2}\dd u\left\{(\nabla_XJ)Y-J(\nabla_{JX}J)Y+J(\nabla_{JY}J)X-(\nabla_YJ)X\right\}
	\end{align*}
	Hence we get
	\begin{align*}
		\dd \sigma (X,Y)
		=&-\frac{1}{4}J^N\dd u\left\{(\nabla_XJ)Y-J(\nabla_{JX}J)Y+J(\nabla_{JY}J)X-(\nabla_YJ)X\right\}\\
		&+\frac{1}{2}J^N\dd u(\nabla_XJ)Y-\frac{1}{2}J^N\dd u(\nabla_YJ)X\\
		=&\frac{1}{4}J^N\dd u\left\{(\nabla_XJ)Y+J(\nabla_{JX}J)Y-(\nabla_YJ)X-J(\nabla_{JY}J)X\right\}.
	\end{align*}
	It follows from Lemma~\ref{lem-b} that $\dd \sigma (X,Y)=J^N\dd u(\NN_J(X,Y))$. By the arbitrary of $X$ and $Y$, we complete the proof.
\end{proof}

On a complete Riemannian manifold $M$ with a pole $x_0$, a radial plane is a plane in $T_xM$ ($x\in M-\{x_0\}$) that contains $\pdv{r}$. By the radial curvature, it means that the restriction of the sectional curvature function to all the radial planes.
\begin{lemma}[cf. {\cite[Lemma~4.5]{MR3149846}}]\label{lem-e}
 	Let $(M,g)$ be a complete Riemannian manifold with a pole $x_0$ and $r$ the distance function relative to $x_0$. Denote by $K_r$ the radial curvature of $M$.
 	\begin{enumerate}[(i)]
 		\item If $K_r\le 0$, then
 		\[
 		\Hess r\ge \frac{1}{r}(g-\dd r\otimes \dd r).
 		\]
 		\item If $K_r\le b^2/(1+r^2)$ with $b^2\in [0,1/4]$, then 
 		\[
 		\Hess r\ge \frac{1+\sqrt{1-4b^2}}{2r}(g-\dd r\otimes \dd r).
 		\]
 		\item If $K_r\le B/(1+r^2)^{1+\varepsilon }$ with $\varepsilon >0$ and $0\le B<2\varepsilon $, then
 		\[
 		\Hess r\ge \frac{1-\frac{B}{2\varepsilon }}{r}(g-\dd r\otimes \dd r).
 		\]
 		\item If $K_r\le -\beta^2$ with $\beta >0$, then 
 		\[
 		\Hess r\ge \beta \coth(\beta r)(g-\dd r\otimes \dd r).
 		\]
 		\item If $K_r\le -a^2/(1+r^2)$ with $a>0$, then
 		\[
 		\Hess r\ge \max\left\{\frac{A}{1+r},\frac{1}{r}\right\}(g-\dd r\otimes \dd r),\quad \mbox{where}\ A:=\frac{1+\sqrt{1+4a^2}}{2}.
 		\]
 	\end{enumerate}
\end{lemma} 

Let $u:(M,J,g)\to (N,J^N,h)$ is a smooth map between almost Hermitian manifolds. Recall the definition of the stress-energy tensor $S_\omega $ in Section~\ref{sub:the_stress_energy_tensors} and replace $\omega $ with $\sigma =\frac{1}{2}(\dd u+J^N\dd uJ)$ to get $S_\sigma =\frac{1}{2}|\sigma |^2g-\sigma \odot \sigma $.
\begin{lemma}\label{lem-f}
	Let $(M,J,g)$ be a complete $2m$-dimensional almost Hermitian manifold with a pole $x_0$ and $r$ the distance function relative to $x_0$, and let $H$ be a symmetric $(0,2)$-tensor on $M$. Denote by $\lambda_1\le \cdots \le \lambda_{2m}$ the eigenvalues of $H$. Then
	\begin{equation}
		\langle S_\sigma ,H\rangle \ge \sum_{\alpha =1}^m(\lambda_i+\lambda_{m+i})|\bar \partial u|^2.
	\end{equation}
\end{lemma}
\begin{proof}
	By the definition of $S_\sigma $, we have
	\begin{align*}
		\langle S_\sigma ,H\rangle =\frac{1}{2}|\sigma |^2\tr_gH-\langle \sigma \odot \sigma ,H\rangle .
	\end{align*}
	We choose a orthonormal basis $\{e_A\}_{A=1}^{2m}=\{e_1,\cdots ,e_m,Je_1,\cdots ,Je_m\}$ at a point such that
	\[
	    H(e_A,e_B)=\lambda_A\delta_{AB}.
	\]
	Then
	\begin{align*}
		\langle \sigma \odot \sigma ,H\rangle 
		=\sum_{A=1}^{2m}\langle \sigma(e_A),\sigma(e_A)\rangle H(e_A,e_A)
		=\sum_{A=1}^{2m}\lambda_A|\sigma(e_A)|^2
		=\sum_{\alpha =1}^m\lambda_i|\sigma(e_\alpha )|^2+\sum_{\alpha =1}^{m}\lambda_{m+i}|\sigma(Je_\alpha )|^2.
	\end{align*}
	Noting that $\sigma $ is $J$-invariant, we have $\sum\limits_{\alpha =1}^m|\sigma(e_\alpha )|=\sum\limits_{\alpha =1}^m|\sigma(Je_\alpha )|=\frac{|\sigma |^2}{2}$. Hence, we deduce
	\begin{align*}
		\langle S_\sigma ,H\rangle 
		=&\frac{1}{2}|\sigma|^2\left(\sum_{\alpha =1}^m\lambda_i+\sum_{\alpha =1}^m\lambda_{m+i}\right)-\left(\sum_{\alpha =1}^m\lambda_i|\sigma(e_\alpha )|^2+\sum_{\alpha =1}^{m}\lambda_{m+i}|\sigma(Je_\alpha )|^2\right)\\
		\ge &\frac{1}{2}|\sigma|^2\left(\sum_{\alpha =1}^m\lambda_i+\sum_{\alpha =1}^m\lambda_{m+i}\right)-\left(\lambda_m\sum_{\alpha =1}^m|\sigma(e_\alpha )|^2+\lambda_{2m}\sum_{\alpha =1}^{m}|\sigma(Je_\alpha )|^2\right)\\
		=&\frac{1}{2}|\sigma |^2\sum_{\alpha =1}^{m-1}(\lambda_i+\lambda_{m+i}).
	\end{align*}
	where we used $\sigma $ is $J$-invariant  . The conclusion follows from $|\sigma |^2=2|\bar \partial u|^2$.
\end{proof}

For two symmetric $(0,2)$-tensor fields $H_1$ and $H_2$ on a manifold, we write $H_1\ge H_2$, meaning that $H_1-H_2$ is nonnegative definite on the whole manifold.
\begin{lemma}\label{lem-g}
	Let $(M,J,g)$ be a complete $2m$-dimensional almost Hermitian manifold with a pole $x_0$ and $r$ the distance function relative to $x_0$. Suppose $f(r)$ be a non-decreasing convex $C^2$ function on $(0,+\infty )$ and $\lambda_1\le \cdots \le \lambda_{2m}$ are the eigenvalues of $\Hess f(r)$. If there exists a positive function $h(r)$ on $(0,+\infty )$ such that 
	\[
	\Hess r\ge h(r)(g-\dd r\otimes \dd r)
	\]
	on $M-\{x_0\}$, then 
	\[
	\sum_{\alpha =1}^m(\lambda_i+\lambda_{m+i})\ge 
	\begin{cases}
		f''(r)+(2m-3)f'(r)h(r) & \mbox{if}\ f'(r)h(r)\ge f''(r),\\
		2(m-1)f'(r)h(r) & \mbox{if}\ f'(r)h(r)<f''(r).
	\end{cases}
	\]
\end{lemma}
\begin{proof}
	By the assumptions, we have
	\begin{align*}
		\Hess f(r)
		=&f''(r)\dd r\otimes \dd r+f'(r)h(r)\Hess r\\
		\ge &f''(r)\dd r\otimes \dd r+f'(r)h(r)(g-\dd r\otimes \dd r)\\
		=&f'(r)h(r)g+[f''(r)-f'(r)h(r)]\dd r\otimes \dd r.
	\end{align*}
	It is easy to see that the eignvalues of $f'(r)h(r)g+[f''(r)-f'(r)h(r)]\dd r\otimes \dd r$ are $f'(r)h(r)$ with multiplicities $2m-1$ and $f''(r)$ with multiplicity $1$. So the conclusion is obvious.
\end{proof}

For the case where the radial curvature satisfies one of (i),(ii) and (iii) in Lemma~\ref{lem-e}, we take $f(r)=\frac{r^2}{2}$ in Lemma~\ref{lem-g} and get
\begin{lemma}[cf. \cite{MR3149846}]\label{lem-h}
	Let $(M,J,g)$ be a $2m$-dimensional complete almost Hermitian manifold with a pole $x_0$. Suppose the radial curvature $K_r$ of $M$ satisfies one of (i), (ii) and (iii) in Lemma~\ref{lem-e}, where $r$ is the distance function relative to $x_0$. Denote by $\mu_1\le \cdots \le \mu_{2m}$ the eigenvalues of $\Hess \frac{r^2}{2}$. Then
	\begin{equation}\label{eq-t}
		\sum_{\alpha =1}^m(\mu_i+\mu_{m+i})\ge 
		\begin{cases}
			2(m-1) & \mbox{if}\ K_r\ \mbox{satisfies}\ (i)\\
			(m-1)(1+\sqrt{1-4b^2}) & \mbox{if}\ K_r\ \mbox{satisfies}\ (ii)\\
			2(m-1)\left(1-\frac{B}{2\varepsilon }\right) & \mbox{if}\ K_r\ \mbox{satisfies}\ (iii).
		\end{cases}
	\end{equation}
\end{lemma}
\begin{proof}
	Notice that $\left(\frac{r^2}{2}\right)'=r, \left(\frac{r^2}{2}\right)''=1$ and $r\cdot \frac{\delta }{r}\le 1$ when $\delta =1,\frac{1+\sqrt{1-4b^2}}{2}$ or $1-\frac{B}{2\varepsilon}$, the conclusion follows directly from Lemma~\ref{lem-g}. 
\end{proof}

When the radial curvature satisfies one of (iv) and (v) in Lemma~\ref{lem-e}, we take $f(r)$ such that it satisfies $f'(r)h(r)\ge f''(r)$. Precisely, we have the following two lemma.
\begin{lemma}\label{lem-i}
	Let $(M,J,g)$ be a $2m$-dimensional complete almost Hermitian manifold with a pole $x_0$. Suppose the radial curvature $K_r$ of $M$ satisfies condition (iv) in Lemma~\ref{lem-e}: $K_r\le -\beta^2$ with $\beta >0$, where $r$ is the distance function relative to $x_0$. Denote by $\xi_1\le \cdots \le \xi_{2m}$ the eigenvalues of $\Hess \cosh(\beta r)$. Then
	\begin{equation}\label{eq-u}
		\sum_{\alpha =1}^m(\xi_i+\xi_{m+i})\ge 2(m-1)\beta^2\cosh(\beta r).
	\end{equation}
\end{lemma}
\begin{proof}
	By Lemma~\ref{lem-e}, (iv), $h(r)=\beta \coth(\beta r)$. Hence we have
	\[
	(\cosh(\beta r))'h(r)=\beta \sinh(\beta r)\cdot \beta \coth(\beta r)=\beta^2\cosh(\beta r)=(\cosh(\beta r))'',
	\]
    and the conclusion follows from Lemma~\ref{lem-g}.
\end{proof}

\begin{lemma}\label{lem-j}
	Let $(M,J,g)$ be a $2m$-dimensional complete almost Hermitian manifold with a pole $x_0$. Suppose the radial curvature $K_r$ of $M$ satisfies condition (v) in Lemma~\ref{lem-e}: $K_r\le -a^2/(1+r^2)$ with $a >0$, where $r$ is the distance function relative to $x_0$. Denote by $\chi_1\le \cdots \le \chi_{2m}$ the eigenvalues of $\Hess \frac{(1+r)^{A+1}}{A+1}$. Then
	\begin{equation}\label{eq-v}
		\sum_{\alpha =1}^m(\chi_i+\chi_{m+i})\ge 
		\max\left\{2(m-1)A(1+r)^{A-1},\ A(1+r)^{A-1}+(2m-3)\frac{(1+r)^A}{r}\right\}
	\end{equation}
\end{lemma}
\begin{proof}
	By Lemma~\ref{lem-e}, (iv), $h(r)=\max\left\{\frac{A}{1+r},\frac{1}{r}\right\}$. (1) If $\frac{A}{1+r}\ge \frac{1}{r}$, then $h(r)=\frac{A}{1+r}$, and then
	\[
	\left(\frac{(1+r)^{A+1}}{A+1}\right)'h(r)=A(1+r)^{A-1}=\left(\frac{(1+r)^{A+1}}{A+1}\right)''.
	\]
	It follows from Lemma~\ref{lem-g} that
	\[
	\sum_{\alpha =1}^m(\chi_i+\chi_{m+i})\ge 2(m-1)A(1+r)^{A-1}.
	\]
	(2) If $\frac{A}{1+r}<\frac{1}{r}$, then $h(r)=\frac{1}{r}$, and then
	\[
	\left(\frac{(1+r)^{A+1}}{A+1}\right)'h(r)=(1+r)^A\cdot \frac{1}{r}=(1+r)^{A-1}\cdot \frac{1+r}{r}>A(1+r)^{A-1}=\left(\frac{(1+r)^{A+1}}{A+1}\right)''
	\]
	It follows from Lemma~\ref{lem-g} that
	\[
	\sum_{\alpha =1}^m(\chi_i+\chi_{m+i})\ge A(1+r)^{A-1}+(2m-3)\frac{(1+r)^A}{r}.
	\]
	The proof is completed.
\end{proof}

Next we prove monotonicity formulae for Hermitian pluriharmonic map into K\"ahler manifolds. When the domain manifold is Hermitian or K\"ahler, the following first theorem is part of \cite[Theorem~3.3]{MR3047049} and \cite[Theorem~4.7]{MR3149846}, respectively, and we write it down for completeness as well. For the rest of this section and the next, we assume that the real dimension of $M$ is at least $4$, i.e., $m\ge 2$.

\begin{theorem}\label{thm-b}
	Let $(M,J,g)$ be a $2m$-dimensional complete almost Hermitian manifold with a pole $x_0$ and $r$ the distance function relative to $x_0$. Suppose the radial curvature $K_r$ of $M$ satisfies one of (i), (ii) and (iii) in Lemma~\ref{lem-e}. Suppose $u:M\to N$ is a Hermitian pluriharmonic map into a K\"ahler manifold $N$ and it satisfies $\Im \NN_J\subset \Ker \dd u$. If $|V|\le \frac{C_1}{r}$, where $C_1<D/2$ and 
    \begin{equation*}
		D=
		\begin{cases}
			2(m-1) & \mbox{if}\ K_r\ \mbox{satisfies}\ (i),\\
			(m-1)(1-\sqrt{1-4b^2}) & \mbox{if}\ K_r\ \mbox{satisfies}\ (ii),\\
			2(m-1)\left(1-\frac{B}{2\varepsilon }\right) & \mbox{if}\ K_r\ \mbox{satisfies}\ (iii),
		\end{cases}
	\end{equation*}
	then
	\begin{equation}\label{eq-w}
		\frac{1}{r_1^\lambda }\int_{B_{r_1}(x_0)}|\bar \partial u|^2\le \frac{1}{r_2^\lambda }\int_{B_{r_2}(x_0)}|\bar \partial u|^2
	\end{equation}
	for any $0<r_1\le r_2$, where $\lambda =D-2C_1>0$.
\end{theorem}
\begin{proof}
	If choose $X=\grad \frac{r^2}{2}=r\grad r$, then $X^\flat =\dd \frac{r^2}{2}=r\dd r$ and $\nabla X^\flat =\Hess \frac{r^2}{2}$. We evaluate each integrand in the integral formula \eqref{eq-r}. We have the following inequality on $\partial B_r(x_0)$,
	\begin{equation}\label{eq-x}
		\begin{split}
			S_\sigma(X,\nu )=&\frac{1}{2}|\sigma |^2\left \langle r\pdv{r},\pdv{r}\right \rangle -\left \langle \sigma\left(r\pdv{r}\right),\sigma\left(\pdv{r}\right)\right \rangle \\
            \le & r|\bar \partial u|^2.
		\end{split}
	\end{equation}
	It follows from Lemmas~\ref{lem-h} and \ref{lem-f}, in which $H=\nabla X^\flat =\Hess \frac{r^2}{2}$, that
    \begin{equation}\label{eq-y}
    	\langle S_\sigma ,\nabla X^\flat \rangle \ge D|\bar \partial u|^2
    \end{equation}
    in $B_r(x_0)$. By Lemmas~\ref{lem-a}, \ref{lem-c} and Lemma~\ref{lem-d} with $\omega =\sigma $, combined with conditions $\Im \NN_J\subset \Ker \dd u$ and $|V|=|V|\le \frac{C}{r}$, we have
    \begin{equation}\label{eq-z}
        \begin{split}
        	(\Div S_\sigma )(X)
        	=&\langle \sigma(V),\sigma(X)\rangle \\
        	\ge &-|V||X||\sigma |^2\\
        	\ge &-2C_1|\bar \partial u|^2
        \end{split} 
    \end{equation}
    in $B_r(x_0)$. By substituting \eqref{eq-x}, \eqref{eq-y} and \eqref{eq-z} into \eqref{eq-r} yields
    \[
    r\int_{\partial B_r(x_0)}|\bar \partial u|^2\ge \lambda \int_{B_r(x_0)}|\bar \partial u|^2.
    \]
    It is obvious from this inequality that
    \[
    \frac{\dv{r}\int_{\partial B_r(x_0)}|\bar \partial u|^2}{\int_{B_r(x_0)}|\bar \partial u|^2}\ge \frac{\lambda }{r}
    \]
    for any $r>0$. By integration over $[r_1,r_2]$, we obtain \eqref{eq-w}.
\end{proof}

\begin{theorem}\label{thm-c}
	Let $(M,J,g)$ be a $2m$-dimensional complete almost Hermitian manifold with a pole $x_0$ and $r$ the distance function relative to $x_0$. Suppose $u:M\to N$ is a Hermitian pluriharmonic map into a K\"ahler manifold $N$ and it satisfies $\Im \NN_J\subset \Ker \dd u$.
	\begin{enumerate}[(1)]
	    \item Suppose the radial curvature $K_r$ of $M$ satisfies $K_r\le -\beta^2$ with $\beta >0$. If $|V|\le \coth(\beta r)C_2$, where $C_2<(m-1)\beta $, then
        \begin{equation}\label{eq-2a}
        	\frac{1}{\sinh^{2(m-1)-2C_2/\beta }(\beta r_1)}\int_{B_{r_1}(x_0)}\cosh(\beta r)|\bar \partial u|^2\le \frac{1}{\sinh^{2(m-1)-2C_2/\beta }(\beta r_2)}\int_{B_{r_2}(x_0)}\cosh(\beta r)|\bar \partial u|^2
        \end{equation}
        for any $0<r_1\le r_2$.
		\item Suppose the radial curvature $K_r$ of $M$ satisfies $K_r\le -a^2/(1+r^2)$ with $a>0$. If $|V|\le \frac{C_3}{r+1}$, where $C_3<(m-1)A$, then
		\begin{equation}\label{eq-2b}
		    \frac{1}{(1+r_1)^{2(m-1)A-2C_3}}\int_{B_{r_1}(x_0)}(1+r)^{A-1}|\bar \partial u|^2\le \frac{1}{(1+r_2)^{2(m-1)A-2C_3}}\int_{B_{r_2}(x_0)}(1+r)^{A-1}|\bar \partial u|^2
		\end{equation}
		for any $0<r_1\le r_2$.
	\end{enumerate}
\end{theorem}
\begin{proof}
    (1) If choose $X=\grad \cosh(\beta r)=\beta \sinh(\beta r)\grad r$, then $X^\flat =\dd \cosh(\beta r)=\beta \cosh(\beta r)\dd r$ and $\nabla X^\flat =\Hess \cosh(\beta r)$. We have 
	\begin{equation*}
		\begin{split}
			S_\sigma(X,\nu )=&\frac{1}{2}|\sigma |^2\left \langle \beta \sinh(\beta r)\pdv{r},\pdv{r}\right \rangle -\left \langle \sigma\left(\beta \sinh(\beta r)\pdv{r}\right),\sigma\left(\pdv{r}\right)\right \rangle \\
            \le & \beta \sinh(\beta r)|\bar \partial u|^2
		\end{split}
	\end{equation*}
	on $\partial B_r(x_0)$. It follows from Lemma~\ref{lem-i} and Lemma~\ref{lem-f}, in which $H=\nabla X^\flat =\Hess \cosh(\beta r)$, that
    \begin{equation*}
    	\langle S_\sigma ,\nabla X^\flat \rangle \ge 2(m-1)\beta^2\cosh(\beta r)|\bar \partial u|^2
    \end{equation*}
    in $B_r(x_0)$. By Lemmas~\ref{lem-a}, \ref{lem-c} and Lemma~\ref{lem-d} with $\omega =\sigma $, combined with conditions $\Im \NN_J\subset \Ker \dd u$ and $|V|\le \coth(\beta r)C_2$, we have
    \begin{equation*}
        \begin{split}
        	(\Div S_\sigma )(X)
        	=&\langle \sigma(V),\sigma(X)\rangle \\
        	\ge &-|V||X||\sigma |^2\\
        	\ge &-2\coth(\beta r)C_2\beta \sinh(\beta r)|\bar \partial u|^2\\
        	= &-2C_2\beta \cosh(\beta r)|\bar \partial u|^2
        \end{split} 
    \end{equation*}
    in $B_r(x_0)$. Substitute these inequalities into \eqref{eq-r}, we get
    \[
    \sinh(\beta r)\int_{\partial B_r(x_0)}|\bar \partial u|^2\ge (2(m-1)\beta-2C_2)\int_{B_r(x_0)}\cosh(\beta r(x))|\bar \partial u|^2. 
    \]
    This implies 
    \begin{align*}
    	\frac{\dv{r}\int_{B_r(x_0)}\cosh(\beta r(x))|\bar \partial u|^2}{\int_{B_r(x_0)}\cosh(\beta r(x))|\bar \partial u|^2}
    	\ge &\frac{(2(m-1)\beta -2C_2)\cosh(\beta r)}{\sinh(\beta r)}\\
        =& (2(m-1)-2C_2/\beta )\frac{\dv{r}\sinh (\beta r)}{\sinh (\beta r)}
    \end{align*}
    for any $r>0$. By integration over $[r_1,r_2]$, we obtain \eqref{eq-2a}.

	(2) If choose $X=\grad \frac{(1+r)^{A+1}}{A+1}=(1+r)^A\grad r$, then $X^\flat =\dd \frac{(1+r)^{A+1}}{A+1}=(1+r)^A\dd r$ and $\nabla X^\flat =\Hess \frac{(1+r)^{A+1}}{A+1}$. We have
	\[
	S_\sigma(X,\nu )\le (1+r)^A|\bar \partial u|^2
	\]
	on $\partial B_r(x_0)$. By Lemmas~\ref{lem-j} and \ref{lem-f} with $H=\nabla X^\flat =\Hess \frac{(1+r)^{A+1}}{A+1}$, we have 
	\[
	\langle S_\sigma ,\nabla X^\flat \rangle \ge 2(m-1)A(1+r)^{A-1}|\bar \partial u|^2
	\]
	in $B_r(x_0)$. By the assumptions we also have
	\[
	\begin{split}
        (\Div S_\sigma )(X)
        =&\langle \sigma(V),\sigma(X)\rangle \\
        \ge &-|V||X||\sigma |^2\\
        \ge &-2C_3(1+r)^{A-1}|\bar \partial u|^2
    \end{split} 
	\]
	in $B_r(x_0)$. Substituting these inequalities into \eqref{eq-r}, we get
	\[
    (1+r)^A\int_{\partial B_r(x_0)}|\bar \partial u|^2\ge (2(m-1)A-2C_3)\int_{B_r(x_0)}(1+r(x))^{A-1}|\bar \partial u|^2.
    \]
    This implies
    \[
    \frac{\dv{r}\int_{B_r(x_0)}(1+r(x))^{A-1}|\bar \partial u|^2}{\int_{B_r(x_0)}(1+r(x))^{A-1}|\bar \partial u|^2}\ge \frac{2(m-1)A-2C_3}{1+r}
    \]
    for any $r>0$. By integration over $[r_1,r_2]$, we obtain \eqref{eq-2b}.
\end{proof}

\begin{remark}
	Theorems~\ref{thm-b} and \ref{thm-c} extend \cite[Theorem~4.7]{MR3149846} and \cite[Theorem~3.3]{MR3047049} to the case of Hermitian pluriharmonic maps originating from almost Hermitian manifolds. Specially, when the domain manifold is Hermitian or K\"ahler, Theorem~\ref{thm-c} improves upon their results under curvature condition (iv) and (v) in Lemma~\ref{lem-e}. additionally, note that as $r\to 0$, $\coth(\beta r)$ and $1/r$ are of the same order, and $\coth(\beta r)>1$ on $(0,+\infty )$. Thus, compared to \cite[Theorem~3.3]{MR3047049}, the assumption on the length of $V$ is more relaxed under curvature condition (iv). In addition, we note that when the domain manifold degenerates into a K\"ahler manifold and its radial curvature satisfies $K_r\le -\beta^2<0$, the corresponding monotonicity formula has already been obtained by Li \cite{MR3459969}.
\end{remark}

\begin{remark}
	Since our ultimate goal is to obtain the holomorphicity of the Hermitian pluriharmonic maps, and for holomorphic maps, we have $\dd u(\NN_J(X,Y))=\NN_{J^N}(\dd u(X),\dd u(Y))$ for any $X,Y\in \Gamma(TM)$. Thus, when the target manifold is a K\"ahler manifold, it naturally follows that $\dd u(\NN_J(X,Y))=0$, giving us reason to assume $\Im \NN_J\subset \Ker \dd u$. Of course, if more specifically $\NN_J=0$, then our theorems concern Hermitian pluriharmonic maps from Hermitian manifolds into K\"ahler manifolds. 
\end{remark}

Now we present an example where $|V|=O(1)$ near infinity. For $|V|=O(1/r)$ near infinity, see \cite[Example~3.8]{MR3047049}, which provides an example on a conformal K\"ahler manifold. In fact, this also applies to conformal semi-K\"ahler manifolds, so it will not be repeated here. From \eqref{eq-2c}, it is evident that conformal semi-K\"ahler manifolds constitute a larger class of almost Hermitian manifolds than conformal K\"ahler manifolds.
\begin{example}
	Let $(M^{2m},J,g)$ be a complete semi-K\"ahler manifold with a pole $x_0$. Denote by $\nabla^g $ and $r(x)$ the $g$-Levi--Civita connection and the $g$-distance function, respectively. Assume $\varphi =\varphi (t)$ is a smooth positive function on $(0,+\infty )$, and consider the conformal metric $\tilde{g}=\varphi^2(r)g$. We know that for $X,Y\in \Gamma(TM)$
	\[
    \nabla^{\tilde{g}}_XY=\nabla^g_XY+X(\ln \varphi )Y+Y(\ln \varphi )X-g(X,Y)\grad_g\ln \varphi ,
	\]
	where $\nabla^{\tilde{g}}$ is the $\tilde{g}$-Levi--Civita connection. From this relation, it is easy to compute that
	\[
	V=-J\delta_{\tilde{g}}J=\frac{2(1-m)\varphi '(r)}{\varphi^3(r)}\grad r.
	\]
	Thus, if $\varphi'(r)<0$, we obtain a complete almost Hermitian manifold $(M,J,\tilde{g})$ on which 
	\[
	|V|_{\tilde{g}}=\frac{2(1-m)\varphi '(r)}{\varphi^4(r)}.
	\]
	We let 
	\[
	\frac{2(1-m)\varphi '(r)}{\varphi^4(r)}=C
	\]
	and integrate this over $[1,r]$ to obtain
	\[
	\varphi (r)=\left(\varphi^{-3}(1)+\frac{3C}{2(m-1)}(r-1)\right)^{-1/3}.
	\] 
	We redefine it as follows (still denoted as $\varphi $):
	\[
	\varphi (r)=
	\begin{cases}
		\varphi(1), & r<1,\\
        \left(\varphi^{-3}(1)+\frac{3C}{2(m-1)}(r-1)\right)^{-1/3}, & r\ge 1.
	\end{cases}
	\]
	We smooth $\varphi (r)$ near $r=1$ to obtain a function $\tilde{\varphi }(r)$ such that
	\[
	\frac{2(1-m)\tilde{\varphi }'(r)}{\tilde{\varphi }^4(r)}\le C.
	\]
	It is clear that the almost Hermitian manifold $(M^{2m},J,\tilde{g}=\tilde{\varphi }^2(r)g)$ satisfies $|V|_{\tilde{g}}\le C$.
\end{example}

Next, we hope to establish monotonicity formulae outside a compact subset of a complete almost Hermitian manifold.
\begin{theorem}\label{thm-d}
	Let $(M,J,g)$ be a $2m$-dimensional complete almost Hermitian manifold with a pole $x_0$ and $r$ the distance function relative to $x_0$. Suppose $u:M\to N$ is a Hermitian pluriharmonic map into a K\"ahler manifold $N$ and it satisfies $\Im \NN_J\subset \Ker \dd u$. 
	\begin{enumerate}[(a)]
	    \item Suppose the radial curvature $K_r$ of $M$ satisfies $K_r\le -\beta^2$ with $\beta >0$. If $|V|\le \coth(\beta r)C_2'$ on $M\setminus B_{R_0}(x_0)$, where $R_0>0$ and $C_2'<(m-1)\beta $, then
        \begin{equation}\label{eq-2d}
        	\frac{1}{\sinh^{2(m-1)-2C_2'/\beta }(\beta r_1)}\int_{B_{r_1}(x_0)\setminus B_{R_0}(x_0)}\cosh(\beta r)|\bar \partial u|^2\le \frac{1}{\sinh^{2(m-1)-2C_2'/\beta }(\beta r_2)}\int_{B_{r_2}(x_0)\setminus B_{R_0}(x_0)}\cosh(\beta r)|\bar \partial u|^2
        \end{equation}
        for any $R_0<r_1\le r_2$. 
		\item Suppose the radial curvature $K_r$ of $M$ satisfies $K_r\le -a^2/(1+r^2)$ with $a>0$. If $|V|\le \frac{C_3'}{r+1}$ on $M\setminus B_{R_0}(x_0)$, where $R_0>0$, $C_3'<C(R_0)/2$ and
		\[
        C(R_0)=
        \begin{cases}
        	2(m-1)A & \mbox{if}\ \frac{AR_0}{1+R_0}\ge 1,\\
        	A+(2m-3)\left(1+\frac{1}{R_0}\right) & \mbox{if}\ \frac{AR_0}{1+R_0}<1,
        \end{cases}
		\]        	
        then
		\begin{equation}\label{eq-2e}
		    \frac{1}{(1+r_1)^{C(R_0)-2C_3'}}\int_{B_{r_1}(x_0)\setminus B_{R_0}(x_0)}(1+r)^{A-1}|\bar \partial u|^2\le \frac{1}{(1+r_2)^{C(R_0)-2C_3'}}\int_{B_{r_2}(x_0)\setminus B_{R_0}(x_0)}(1+r)^{A-1}|\bar \partial u|^2
		\end{equation}
		for any $R_0<r_1\le r_2$. 
	\end{enumerate}
\end{theorem}
\begin{proof}
	(a) Take $X=\grad \cosh(\beta r)=\beta \sinh(\beta r)\grad r$. For any $r>R_0$, we let $D=B_r(x_0)\setminus B_{R_0}(x_0)$ and apply \eqref{eq-r} on $D$.  First, using Lemmas~\ref{lem-f} and \ref{lem-g}, we have
	\begin{align*}
		\int_{\partial D}S_\sigma (X,\nu )
		=&\int_{\partial B_r(x_0)}S_\sigma \left(X,\pdv{r} \right)-\int_{\partial B_{R_0}(x_0)}S_\sigma \left(X,\pdv{r} \right)\\
		=&\beta \sinh(\beta r)\int_{\partial B_r(x_0)}S_\sigma \left(\pdv{r},\pdv{r} \right)-\beta \sinh(\beta R_0)\int_{\partial B_{R_0}(x_0)}S_\sigma \left(\pdv{r},\pdv{r} \right)\\
		\le &\beta \sinh(\beta r)\int_{\partial B_r(x_0)}S_\sigma \left(\pdv{r},\pdv{r} \right)\\
		\le &\beta \sinh(\beta r)\int_{\partial B_r(x_0)}|\bar \partial u|^2.
	\end{align*}
	since $S_\sigma (\nu ,\nu )=\frac{|\sigma |^2}{2}-\left \langle \sigma\left(\pdv{r}\right),\sigma\left(\pdv{r}\right)\right \rangle \ge 0$. On the other hand, by \eqref{eq-r} and the proof of Theorem~\ref{thm-c}, we have 
	\begin{align*}
		\int_{\partial D}S_\sigma (X,\nu )
		=&\int_{B_r(x_0)\setminus B_{R_0}(x_0)}\left\{\langle S_\sigma ,\nabla X^\flat \rangle +(\Div S_\sigma )(X)\right\}\\
		\ge &\left(2(m-1)\beta^2-2C_2'\beta \right)\int_{B_r(x_0)\setminus B_{R_0}(x_0)}\cosh(\beta r(x))|\bar \partial u|^2
	\end{align*}
	By combining these two inequalities, we obtain
	\begin{align*}
		\sinh(\beta r)\int_{\partial B_r(x_0)}|\bar \partial u|^2\ge \left(2(m-1)\beta -2C_2'\right)\int_{B_r(x_0)\setminus B_{R_0}(x_0)}\cosh(\beta r(x))|\bar \partial u|^2.
	\end{align*}
	It follows that
	\begin{align*}
    	\frac{\dv{r}\int_{B_r(x_0)\setminus B_{R_0}(x_0)}\cosh(\beta r(x))|\bar \partial u|^2}{\int_{B_r(x_0)\setminus B_{R_0}(x_0)}\cosh(\beta r(x))|\bar \partial u|^2}
    	\ge &\frac{(2(m-1)\beta -2C_2')\cosh(\beta r)}{\sinh(\beta r)}\\
        =& (2(m-1)-2C_2'/\beta )\frac{\dv{r}\sinh (\beta r)}{\sinh (\beta r)}
    \end{align*}
    for any $r>R$. By integration over $[r_1,r_2]$, we obtain \eqref{eq-2d}.

    (b) Take $X=\grad \frac{(1+r)^{A+1}}{A+1}$, apply \eqref{eq-r} on $D=B_r(x_0)\setminus B_{R_0}(x_0)$. The remaining steps of the proof are similar to the proof of (a), where the value of $C(R_0)$ only requires attention to Lemma~\ref{lem-j}. 
\end{proof}


\section{Holomorphicity and constancy of Hermitian pluriharmonic maps} 
\label{sec:holomorphicity_of_hermitian_pluriharmonic_maps}

In this section, we use the monotonicity formula, under certain growth assumptions on the partial energy $E''$, to obtain the holomorphicity of Hermitian pluriharmonic maps. 

\begin{theorem}\label{thm-e}
	Under the conditions of Theorem~\ref{thm-b}, furthermore, if the partial energy of $u$ satisfies 
	\begin{equation}\label{eq-2f}
		\int_{B_r(x_0)}|\bar \partial u|^2=o(r^\lambda )\quad \mbox{as}\ r\to \infty ,
	\end{equation}
	then $u$ is holomorphic. 
\end{theorem}
\begin{proof}
	From Theorem~\eqref{thm-b} and condition~\eqref{eq-2f}, in monotonicity formula \eqref{eq-w}, letting $r_2\to \infty $, we immediately obtain that $u$ is holomorphic.
\end{proof}

\begin{theorem}\label{thm-f}
	Let $(M,J,g)$ be a $2m$-dimensional complete almost Hermitian manifold with a pole $x_0$ and $r$ the distance function relative to $x_0$. Suppose $u:M\to N$ is a Hermitian pluriharmonic map into a K\"ahler manifold $N$ and it satisfies $\Im \NN_J\subset \Ker \dd u$.
	\begin{enumerate}[(1)]
	    \item Suppose the radial curvature $K_r$ of $M$ satisfies $K_r\le -\beta^2$ with $\beta >0$. If $|V|\le \coth(\beta r)C_2$, where $C_2<(m-1)\beta $, and
        \begin{equation}\label{eq-2g}
        	\int_{B_r(x_0)}|\bar \partial u|^2=o\left(e^{[(2m-3)\beta -2C_2]r}\right)\quad \mbox{as}\ r\to \infty ,
        \end{equation}
        then $u$ is holomorphic.
		\item Suppose the radial curvature $K_r$ of $M$ satisfies $K_r\le -a^2/(1+r^2)$ with $a>0$. If $|V|\le \frac{C_3}{r+1}$, where $C_3<(m-1)A$, and
		\begin{equation}\label{eq-2h}
		    \int_{B_r(x_0)}|\bar \partial u|^2=o\left((1+r)^{(2m-3)A+1-2C_3}\right)\quad \mbox{as}\ r\to \infty ,
		\end{equation}
		then $u$ is holomorphic.
	\end{enumerate}
\end{theorem}
\begin{proof}
	(1) Noting that the right-hand side of monotonicity formula \eqref{eq-2a} can be enlarged to
    \begin{align*}
    	\frac{1}{\sinh^{2(m-1)-2C_2/\beta }(\beta r_2)}\int_{B_{r_2}(x_0)}\cosh(\beta r)|\bar \partial u|^2
    	\le &\frac{\cosh(\beta r_2)}{\sinh^{2(m-1)-2C_2/\beta }(\beta r_2)}\int_{B_{r_2}(x_0)}|\bar \partial u|^2\\
    	= &\coth(\beta r_2)\left(\frac{e^{\beta r_2}}{\sinh(\beta r_2)}\right)^{e^{2m-3-2C_2/\beta }}\frac{\int_{B_{r_2}(x_0)}|\bar \partial u|^2}{e^{[(2m-3)\beta -2C_2]r}},
    \end{align*}
    and then letting $r_2\to \infty $, combined with condition~\eqref{eq-2g}, we thus prove that $u$ is holomorphic.

    (2) Noting that the right-hand side of monotonicity formula \eqref{eq-2b} can be enlarged to
    \begin{align*}
    	\frac{1}{(1+r_2)^{2(m-1)A-2C_3}}\int_{B_{r_2}(x_0)}(1+r)^{A-1}|\bar \partial u|^2
    	\le &\frac{(1+r_2)^{A-1}}{(1+r_2)^{2(m-1)A-2C_3}}\int_{B_{r_2}(x_0)}|\bar \partial u|^2\\
    	= &\frac{1}{(1+r_2)^{(2m-3)A+1-2C_3}}\int_{B_{r_2}(x_0)}|\bar \partial u|^2,
    \end{align*}
    and then letting $r_2\to \infty $, combined with condition~\eqref{eq-2h}, we thus prove that $u$ is holomorphic.
\end{proof}

\begin{remark}
	When the domain manifold is K\"ahler or Hermitian, Theorem~\ref{thm-f} improves the corresponding results in \cite{MR3149846} and \cite{MR3047049} under conditions $K_r\le -\beta^2<0$ and $K_r\le -\frac{a^2}{1+r^2}<0$. For instance, in the K\"ahler case (one has $C_3=0$), as $r\to \infty $, \cite{MR3149846} requires polynomial growth of the partial energy under the former condition, and $o(r^{\Lambda_0})$ growth under the latter condition, where $\Lambda_0$ is any positive number less than $1+(2m-3)A$.
\end{remark}

\begin{remark}
	Starting from the monotonicity formula that holds outside a compact subset, corresponding to the cases of (a) and (b) in Theorem~\ref{thm-d}, under conditions \eqref{eq-2g} and \eqref{eq-2h} respectively, it can be proven that $u$ is holomorphic. In fact, similar to the proof of Theorem~\ref{thm-f}, we can derive that $u$ is holomorphic in the open set $M\setminus B_{R_0}(x_0)$, which is equivalent to $\sigma =0$ in $M\setminus B_{R_0}(x_0)$. Furthermore, by Lemmas~\ref{lem-a} and \ref{lem-c}, we obtain that
	\begin{align*}
		|\dd \sigma |^2+|\delta \sigma |^2
		=|\sigma (V)|^2\le |V|^2_K|\sigma |^2
	\end{align*} 
	holds on any compact subset of $M$. Thus, according to the unique continuation property in \cite{MR140031}, $\sigma \equiv 0$ on $M$, meaning that $u$ is holomorphic on the whole $M$. 
\end{remark}

\begin{remark}
	In the two theorems above, if those growth conditions are applied separately to the partial energy $E'$, the corresponding anti-holomorphicity can be obtained. As a corollary, if those growth conditions are applied separately to the energy $E$, the corresponding constancy can be obtained.
\end{remark}

\begin{remark}
	Noticing that Hermitian pluriharmonic maps are $V$-harmonic (Remark~\ref{rem-c}), using the method from \cite{MR895408}, under the appropriate length condition for $V$ and some radial curvature pinching conditions, one can also derive some monotonicity formulae for Hermitian harmonic maps into a K\"ahler manifold. Furthermore, under suitable energy growth assumptions, constancy results can be obtained.
\end{remark}



\begin{thebibliography}{10}

\bibitem{MR140031}
N.~Aronszajn, A.~Krzywicki, and J.~Szarski.
\newblock A unique continuation theorem for exterior differential forms on
  {R}iemannian manifolds.
\newblock {\em Ark. Mat.}, 4:417--453 (1962), 1962.

\bibitem{MR1301014}
J.~Chen.
\newblock A boundary value problem for {H}ermitian harmonic maps and
  applications.
\newblock {\em Proc. Amer. Math. Soc.}, 124(9):2853--2862, 1996.

\bibitem{MR2995205}
Q.~Chen, J.~Jost, and H.~Qiu.
\newblock Existence and {L}iouville theorems for {$V$}-harmonic maps from
  complete manifolds.
\newblock {\em Ann. Global Anal. Geom.}, 42(4):565--584, 2012.

\bibitem{MR3427131}
Q.~Chen, J.~Jost, and G.~Wang.
\newblock A maximum principle for generalizations of harmonic maps in
  {H}ermitian, affine, {W}eyl, and {F}insler geometry.
\newblock {\em J. Geom. Anal.}, 25(4):2407--2426, 2015.

\bibitem{MR3149846}
Y.~Dong.
\newblock Monotonicity formulae and holomorphicity of harmonic maps between
  {K}\"{a}hler manifolds.
\newblock {\em Proc. Lond. Math. Soc. (3)}, 107(6):1221--1260, 2013.

\bibitem{MR495450}
J.~Eells and L.~Lemaire.
\newblock A report on harmonic maps.
\newblock {\em Bull. London Math. Soc.}, 10(1):1--68, 1978.

\bibitem{MR1456265}
P.~Gauduchon.
\newblock Hermitian connections and {D}irac operators.
\newblock {\em Boll. Un. Mat. Ital. B (7)}, 11(2):257--288, 1997.

\bibitem{MR184185}
A.~Gray.
\newblock Minimal varieties and almost {H}ermitian submanifolds.
\newblock {\em Michigan Math. J.}, 12:273--287, 1965.

\bibitem{MR2115446}
H.-C. Grunau and M.~K\"{u}hnel.
\newblock On the existence of {H}ermitian-harmonic maps from complete
  {H}ermitian to complete {R}iemannian manifolds.
\newblock {\em Math. Z.}, 249(2):297--327, 2005.

\bibitem{MR1226528}
J.~Jost and S.-T. Yau.
\newblock A nonlinear elliptic system for maps from {H}ermitian to {R}iemannian
  manifolds and rigidity theorems in {H}ermitian geometry.
\newblock {\em Acta Math.}, 170(2):221--254, 1993.

\bibitem{MR3459969}
J.~Li.
\newblock Monotonicity formulae and vanishing theorems.
\newblock {\em Pacific J. Math.}, 281(1):125--136, 2016.

\bibitem{li-zhang}
J.~Li and X.~Zhang.
\newblock The strong rigidity of compact almost-{K}ähler manifolds (in
  chinese).
\newblock {\em Sci Sin Math}, 48(10):1355--1370, 2018.

\bibitem{MR2296630}
Z.~Li and X.~Zhang.
\newblock Hermitian harmonic maps into convex balls.
\newblock {\em Canad. Math. Bull.}, 50(1):113--122, 2007.

\bibitem{MR3275649}
K.~Liu and X.~Yang.
\newblock Hermitian harmonic maps and non-degenerate curvatures.
\newblock {\em Math. Res. Lett.}, 21(4):831--862, 2014.

\bibitem{MR1718630}
L.~Ni.
\newblock Hermitian harmonic maps from complete {H}ermitian manifolds to
  complete {R}iemannian manifolds.
\newblock {\em Math. Z.}, 232(2):331--355, 1999.

\bibitem{MR833809}
J.~H. Sampson.
\newblock Applications of harmonic maps to {K}\"{a}hler geometry.
\newblock In {\em Complex differential geometry and nonlinear differential
  equations ({B}runswick, {M}aine, 1984)}, volume~49 of {\em Contemp. Math.},
  pages 125--134. Amer. Math. Soc., Providence, RI, 1986.

\bibitem{MR584075}
Y.~T. Siu.
\newblock The complex-analyticity of harmonic maps and the strong rigidity of
  compact {K}\"{a}hler manifolds.
\newblock {\em Ann. of Math. (2)}, 112(1):73--111, 1980.

\bibitem{MR658472}
Y.~T. Siu.
\newblock Complex-analyticity of harmonic maps, vanishing and {L}efschetz
  theorems.
\newblock {\em J. Differential Geometry}, 17(1):55--138, 1982.

\bibitem{MR1391729}
Y.~Xin.
\newblock {\em Geometry of harmonic maps}, volume~23 of {\em Progress in
  Nonlinear Differential Equations and their Applications}.
\newblock Birkh\"{a}user Boston, Inc., Boston, MA, 1996.

\bibitem{MR851976}
Y.~L. Xin.
\newblock Differential forms, conservation law and monotonicity formula.
\newblock {\em Sci. Sinica Ser. A}, 29(1):40--50, 1986.

\bibitem{MR895408}
Y.~L. Xin.
\newblock Liouville type theorems and regularity of harmonic maps.
\newblock In {\em Differential geometry and differential equations ({S}hanghai,
  1985)}, volume 1255 of {\em Lecture Notes in Math.}, pages 198--208.
  Springer, Berlin, 1987.

\bibitem{MR3047049}
G.~Yang, Y.~Han, and Y.~Dong.
\newblock Partial energies monotonicity and holomorphicity of {H}ermitian
  pluriharmonic maps.
\newblock {\em Sci. China Math.}, 56(5):1019--1032, 2013.

\bibitem{MR3077216}
X.~Zhang.
\newblock Hermitian harmonic maps between almost {H}ermitian manifolds.
\newblock In {\em Recent developments in geometry and analysis}, volume~23 of
  {\em Adv. Lect. Math. (ALM)}, pages 485--493. Int. Press, Somerville, MA,
  2012.

\end{thebibliography}
\end{document}